\documentclass[12pt]{amsart}
\usepackage{times,amsfonts,amsmath,amstext,amsbsy,amssymb,amsopn,amsthm,upref,eucal,amscd,graphicx}
\usepackage[T1]{fontenc}
\usepackage{color}

\numberwithin{equation}{section}

\newtheorem{theorem}{Theorem}

\newtheorem{lemma}{Lemma}[section]

\newtheorem{corollary}{Corollary}[section]

\theoremstyle{definition}

\theoremstyle{remark}

\newtheorem{remark}{Remark}[section]

\renewcommand{\emptyset}{\varnothing}

\renewcommand{\epsilon}{\varepsilon}
\renewcommand{\phi}{\varphi}
\renewcommand{\kappa}{\varkappa}

\newlength{\halfbls}

\begin{document}
\title[Explicit Ratner's estimates]{Some quantitative versions of Ratner's mixing estimates}
\author{Carlos Matheus}
\address{Carlos Matheus: Universit\'e Paris 13, Sorbonne Paris Cit\'e, LAGA, CNRS (UMR 7539), F-93430, Villetaneuse, France.}
\email{matheus@impa.br.}

\date{December 4, 2013}

\begin{abstract} We give explicit versions for some of Ratner's estimates on the decay of matrix coefficients of $SL(2,\mathbb{R})$-representations.
\end{abstract}
\subjclass[2000]{Primary: 37D40 (Dynamical systems of geometric origin and hyperbolicity); Secondary: 37A25 (Ergodicity, mixing, rates of mixing), 06B15 (Representation theory)}
\keywords{Geodesic flows, hyperbolic surfaces, Ratner's estimates of rate of mixing, quantitative versions of Ratner's mixing estimates.}
\maketitle
\vspace*{-0.5cm}

\setcounter{tocdepth}{2}
\tableofcontents


\section{Introduction}\label{s.introduction}

The study of topological and ergodic features of geodesic and horocycle flows is a classical subject in Dynamical Systems with applications in other fields of Mathematics. For example, the topological features of horocycle flows were used by G. Margulis \cite{Margulis} to establish the Oppenheim conjecture in Number Theory, and, more recently, the ergodic properties (namely, exponential mixing) of geodesic flows on hyperbolic manifolds were successfully applied by J. Kahn and V. Markovic in their work \cite{Kahn:Markovic} on essential immersed hyperbolic surfaces inside closed hyperbolic $3$-manifolds.  

On the other hand, given the nature of the usual topological and ergodic-theoretical results, it is not surprising that most applications of geodesic and horocycle flows to other areas are \emph{qualitative} in the sense that some asymptotic behavior is assured but no rate of convergence is provided. Of course, while qualitative information normally suffices in most applications, sometimes this is not the case in certain fields (such as Number Theory). Hence, it is not rare that \emph{quantitative} versions of qualitative dynamical results are necessary. In particular, this provides part of the motivation behind certain quantitative versions of equidistribution results such as the recent theorem of M. Einsiedler, G. Margulis and A. Venkatesh \cite{EMV}.

In this note, we will discuss some quantitative versions of M. Ratner's estimates of the rate of mixing of geodesic flows \cite{Rt}. In fact, it was known among experts that all quantities in M. Ratner's article \cite{Rt} could be rendered explicit. Thus, in some sense, her original paper was already providing quantitative information about geodesic flows. In particular, we do \emph{not} claim originality in the present note. On the other hand, the author is not aware of accessible references in the literature where explicit versions of Ratner's estimates are discussed. Hence, he believes that this note might be helpful in certain applications of Ratner's mixing estimates. Indeed, this note was originally written as part of a paper by G. Schmith\"usen and the author \cite{Matheus:Schmithuesen} where quantitative versions of Ratner's results were used to exhibit explicit rational points in the moduli spaces of Abelian differentials generating Teichm\"uller curves with complementary series. Ultimately the quantitative Ratner estimates were replaced by applications of Cheeger-Buser inequalities in the \cite{Matheus:Schmithuesen} paper, so the present author made the note about the quantitative Ratner estimates publicly available on his weblog \cite{DM}. A year later, the author was contacted by Han Li who communicated that the discussion in the informal notes \cite{DM} were also naturally related to a forthcoming paper \cite{Li:Margulis} by Han Li and Gregory Margulis (where they study the $3$-dimensional Markov spectrum and they largely improve a recent result of A. Mohammadi \cite{Mohammadi}). For these reasons, in order to make these estimates more accessible for future work of others, the author has formalized the results on quantitative Ratner estimates in this note.


Let us now briefly describe the organization of this note. In the next section, we recall some elementary aspects of the representation theory of $SL(2,\mathbb{R})$, and we state quantitative versions of some results in \cite{Rt}, cf. Theorems \ref{t.theorem1} and \ref{t.theorem3} (and also Corollary \ref{c.ratner} below). Then, in the two subsequent sections, we follow closely the arguments in \cite{Rt} to render all implied constants in Lemma 2.2 in Ratner's article \cite{Rt} as explicitly as possible, and, in particular, we will summarize our conclusions in Lemma \ref{l.lemma2-2} below. Finally, in the last section, we apply Lemma \ref{l.lemma2-2} and Ratner's arguments in \cite{Rt} to complete the proof of Theorems \ref{t.theorem1} and \ref{t.theorem3}.

\subsection*{Acknowledgments} The author is thankful to the anonymous referee, Idris Assani and Kimberly Presser for their immense help in improving previous versions of this note. The author was partially supported by the Balzan project of Jacob Palis and by the French ANR grant ``GeoDyM'' (ANR-11-BS01-0004).

\section{Preliminaries and main statements}\label{s.preliminaries}

In this section, we briefly review some basic facts about the representation theory of $SL(2,\mathbb{R})$. The reader may consult A. Knapp's book \cite{Knapp} for the proofs of the results mentioned below.

Let $T: SL(2,\mathbb{R})\to U(\mathcal{H})$ be an unitary representation of $SL(2,\mathbb{R})$, i.e., $T$ is a homomorphism from $SL(2,\mathbb{R})$ into the group $U(\mathcal{H})$ of unitary transformations of the complex separable Hilbert space $\mathcal{H}=\mathcal{H}(T)$. We say that a vector $v\in\mathcal{H}$ is a $C^k$-vector of $T$ if $g\mapsto T(g)v$ is $C^k$. Recall that the subset of $C^{\infty}$-vectors is dense in $\mathcal{H}$. 

The Lie algebra $sl(2,\mathbb{R})$ of $SL(2,\mathbb{R})$ (i.e., the tangent space of $SL(2,\mathbb{R})$ at the identity element) is the set of all $2\times2$ matrices with zero trace. Given a $C^1$-vector $v$ of $T$ and $X\in sl(2,\mathbb{R})$, the Lie derivative $L_X v$ is
$$L_X v := \lim\limits_{t\rightarrow0}\frac{T(\exp(tX))\cdot v - v}{t}$$
where $\exp(X)$ is the exponential map (of matrices).

An important basis of $sl(2,\mathbb{R})$ is
$$W:=\left(\begin{array}{cc}0&1\\-1&0\end{array}\right), \quad Q:=\left(\begin{array}{cc}1 & 0 \\ 0&-1\end{array}\right), \quad V:=\left(\begin{array}{cc}0&1\\1&0\end{array}\right)$$
This basis has the property that 
$$\exp(tW)=\left(\begin{array}{cc}\cos t&\sin t\\-\sin t&\cos t\end{array}\right):=r(t),$$ 
$$\exp(tQ) = \left(\begin{array}{cc}e^t & 0\\ 0 &e^{-t}\end{array}\right):=a(t)$$ and 
$$\exp(tV) = \left(\begin{array}{cc}\cosh t&\sinh t\\-\sinh t&\cosh t\end{array}\right),$$ and, furthermore, $[Q,W]=2V$, $[Q,V]=2W$ and $[W,V]=2Q$ where $[.,.]$ is the Lie bracket of $sl(2,\mathbb{R})$ (i.e., $[A,B]:= AB-BA$ is the commutator). 

The Casimir operator $\Omega_{T}$ is $\Omega_{T}:=(L_V^2+L_Q^2-L_W^2)/4$ on the dense subspace of $C^2$-vectors of $T$. It is known that $\langle \Omega_{T}v,w\rangle = \langle v,\Omega_{T}w\rangle$ for any $C^2$-vectors $v,w\in\mathcal{H}$, the closure of $\Omega_{T}$ is self-adjoint, $\Omega_{T}$ commutes with $L_X$ on $C^3$-vectors for any $X\in sl(2,\mathbb{R})$ and $\Omega_{T}$ commutes with $T(g)$ for any $g\in SL(2,\mathbb{R})$.

Furthermore, when the representation $T$ is irreducible, $\Omega_{T}$ is a scalar multiple of the identity operator, i.e., $\Omega_{T}v = \lambda(T)v$ for some $\lambda(T)\in\mathbb{R}$ and for any $C^2$-vector $v\in\mathcal{H}$ of $T$.

Also, given $p\geq 0$, we will denote by $K(T,p)$ the set of vectors $v\in\mathcal{H}$ such that $\theta\mapsto T(r(\theta))v$ is $C^p$.

Finally, for later use, we need to introduce the following explicit constants and functions. First, we define
$$C_1:=(1-e^{-4})^{-1}, \quad C_2:=\frac{2}{1-e^{-4}}\left(1+\frac{2}{e^2(1-e^{-4})}+\frac{2}{e^4(1-e^{-4})}\right).$$
Secondly, using these constants we can define the following functions of the parameter $\lambda\in\mathbb{R}$:
$$\bar{K}_\lambda=\left\{\begin{array}{cl} 4C_1/9e^3+2C_2/e+e & \textrm{if }\lambda\leq-1/4\\ 
4C_1/9e^3+2C_2/e+e & \textrm{if }-1/4<\lambda<0\\ 
(C_1+C_2)/2 & \textrm{if } 0\leq\lambda\end{array}\right.,$$
$$\tilde{K}_{\lambda}=\left\{\begin{array}{cl} (1+2\sqrt{2})e + (32+\sqrt{2})C_1^2/3e^3  & \textrm{if }\lambda\leq-1/4\\ 
3e + e^2 + 4C_1/9e^3 & \textrm{if }-1/4<\lambda<0\\ 
e^2& \textrm{if } 0\leq\lambda\end{array}\right..$$
Then, we consider the following auxiliary function of the parameters $\lambda, t\in\mathbb{R}$:
$$b_{\lambda}(t)=\left\{\begin{array}{cc} te^{-t}, & \textrm{if } \lambda\leq-1/4 \\
te^{(-1+\sqrt{1+4\lambda})t}, & \textrm{if } -1/4<\lambda<0 \\ 
te^{-2t}, & \textrm{if } 0\leq\lambda \end{array}\right.$$

Once we dispose of these notations, we are ready to state quantitative versions of some theorems in M. Ratner's paper \cite{Rt}. We start with the following two theorems providing explicit analogues to Theorems 1 and 3 (resp.) in \cite{Rt}.

\begin{theorem}\label{t.theorem1}Let $T$ be a non-trivial irreducible unitary representation of $SL(2,\mathbb{R})$ in $\mathcal{H}(T)$ and let $\lambda=\lambda(T)$. Let $v,w\in K(T,3)$ and $B(t)=\langle v,w\circ a(t)\rangle$. Then, for all $t\geq 1$, 
\begin{eqnarray*}
|B(t)|&\leq&\sqrt{2\zeta(2)}\cdot\bar{K}_{\lambda}\cdot\|L_W^3v\|\cdot(\|w\|+\sqrt{2\zeta(6)}\|L_W^3w\|)\cdot b_{\lambda}(t) \\
&+&\sqrt{2\zeta(2)}\cdot\bar{K}_{\lambda}\cdot(\|v\|+\sqrt{2\zeta(6)}\|L_W^3v\|)\cdot\|L_W^3w\|\cdot b_{\lambda}(t)\\
&+&\tilde{K}_{\lambda}\cdot(\|v\|+\sqrt{2\zeta(6)}\|L_W^3v\|)\cdot(\|w\|+\sqrt{2\zeta(6)}\|L_W^3w\|)\cdot b_{\lambda}(t)
\end{eqnarray*}
\end{theorem}

\begin{theorem}\label{t.theorem3}Let $T$ be an unitary representation of $SL(2,\mathbb{R})$ having no non-zero invariant vectors in $\mathcal{H}(T)$. Denote by $\Lambda=\Lambda(\Omega_T)$ the spectrum of the Casimir operator and
$$A(T)=\Lambda\cap(-1/4,0).$$
If $A(T)\neq\emptyset$, let $\beta(T)=\sup A(T)$ and $\sigma(T)=-1+\sqrt{1+4\beta(T)}$. Assume that $\beta(T)<0$ when $A(T)\neq\emptyset$. Let $B(t)=\langle v,w\circ a(t)\rangle$ with $v,w\in K(T,3)$. Then, for all $t\geq 1$,
\begin{eqnarray*}
|B(t)|&\leq&\sqrt{2\zeta(2)}\cdot\bar{K}\cdot\|L_W^3v\|\cdot(\|w\|+\sqrt{2\zeta(6)}\|L_W^3w\|)\cdot b_T(t) \\
&+&\sqrt{2\zeta(2)}\cdot\bar{K}\cdot(\|v\|+\sqrt{2\zeta(6)}\|L_W^3v\|)\cdot\|L_W^3w\|\cdot b_T(t)\\
&+&\tilde{K}_T\cdot(\|v\|+\sqrt{2\zeta(6)}\|L_W^3v\|)\cdot(\|w\|+\sqrt{2\zeta(6)}\|L_W^3w\|)\cdot b_T(t)
\end{eqnarray*}
where $\bar{K}=\bar{K}_{\beta(T)}$ and $\tilde{K}_T=\tilde{K}_{\beta(T)}$ and $b_T(t)=b_{\beta(T)}(t)$.
\end{theorem}

Next, let us recall that Ratner's theorems in \cite{Rt} have nice consequences to the study of rates of mixing of the geodesic flow on hyperbolic surfaces. More precisely, we consider the regular representation of $SL(2,\mathbb{R})$ on $L^2(S)$ where $S=SO(2,\mathbb{R})\backslash SL(2,\mathbb{R})/\Gamma=\mathbb{H}/\Gamma$ is a hyperbolic surface of finite area (i.e., $\Gamma$ is a lattice of $SL(2,\mathbb{R})$). Then, by noticing that the lift to the unit tangent bundle $T^1S=SL(2,\mathbb{R})/\Gamma$ of $S$ of a function $L^2(S)$ is constant along the orbits of $SO(2,\mathbb{R})$, one has that the Lie derivative $L_W$ of such lifts vanish. Therefore, since a number $\lambda\in(-1/4,0)$ belongs to the spectrum of the Casimir operator if and only if it belongs to the spectrum of the hyperbolic Laplacian $\Delta_{\Gamma}$ on $S=\mathbb{H}/\Gamma$, by direct application of Theorem \ref{t.theorem3} above, one gets the following corollary giving a quantitative version of (part of) Theorem 2 in Ratner's paper \cite{Rt}. 

\begin{corollary}\label{c.ratner}Let $\Gamma$ be a lattice of $SL(2,\mathbb{R})$ and let $T=T_{\Gamma}$ be the regular representation of $SL(2,\mathbb{R})$ on $L^2(S)$ where $S=SO(2,\mathbb{R})\backslash SL(2,\mathbb{R})/\Gamma=\mathbb{H}/\Gamma$. Given $v,w\in L^2(S)$ with $\int_S vd\mu=\int_S wd\mu=0$, it holds 
$$|\langle v,T(a(t))w\rangle|:=|\langle v,w\circ a(t)\rangle|\leq\tilde{K}_{\Gamma}\cdot\|v\|_{L^2(S)}\cdot\|w\|_{L^2(S)}\cdot b_{\Gamma}(t)$$
where 
$$\tilde{K}_{\Gamma}=\left\{\begin{array}{cl}\frac{(32+\sqrt{2})C_1^2}{3e^3}+(1+2\sqrt{2})e & \textrm{ if } \lambda_1(\Delta_{\Gamma})\leq-1/4, \\ 
\frac{4C_1}{9e^3}+3e+e^2 & \textrm{ if } -1/4<\lambda_1(\Delta_{\Gamma})<0\end{array}\right.,$$
$$b_{\Gamma}(t)=\left\{\begin{array}{cl}t\cdot e^{-t} & \textrm{ if } \lambda_1(\Delta_{\Gamma})\leq-1/4, \\ 
t\cdot e^{\sigma(\Gamma)t} & \textrm{ if } -1/4<\lambda_1(\Delta_{\Gamma})<0\end{array}\right.,$$
$\Delta_{\Gamma}$ is the hyperbolic Laplacian on $S=\mathbb{H}/\Gamma$, $\lambda_1(\Delta_{\Gamma})$ is its first eigenvalue, $\sigma(\Gamma)=-1+\sqrt{1+4\lambda_1(\Delta_{\Gamma})}$ is the size of the spectral gap (if $\lambda_1(\Delta_{\Gamma})\in(-\frac{1}{4},0)$), and the constants $C_1,C_2>0$ as above.
\end{corollary}

\begin{remark} It is worth it to point out that the explicit constants appearing in these quantitative versions of Ratner's estimates are not very large. For instance, since $C_1=(1-e^{-4})^{-1}$, we have that the constant $\tilde{K}_{\Gamma}$ in Corollary \ref{c.ratner} above satisfies 
$$\tilde{K}_{\Gamma}\leq \frac{(32+\sqrt{2})}{3e^3(1-e^{-4})^2}+(1+2\sqrt{2})e < 10.9822$$
\end{remark}

Our main goal is to prove Theorems \ref{t.theorem1} and \ref{t.theorem3}. For this reason, we will spend the next two sections performing several preliminary estimates to derive a quantitative version (namely, Lemma \ref{l.lemma2-2} below) of a key estimate in Ratner's arguments (namely, Lemma 2.2 in \cite{Rt}).

\section{Some preparatory estimates}

Let $T$ be a \emph{non-trivial} irreducible unitary $SL(2,\mathbb{R})$-representation in a complex separable Hilbert space $\mathcal{H} = \mathcal{H}(T)$. We define, for each $n\in\mathbb{Z}$,
$$\mathcal{H}_n(T) = \{v\in\mathcal{H}(T): T(r(\theta))v=e^{in\theta}v \,\forall\,\theta\in\mathbb{R}\}.$$
where $r(\theta) = \left(\begin{array}{cc}\cos \theta & \sin\theta \\ -\sin\theta & \cos\theta\end{array}\right)\in SO(2,\mathbb{R})$, $\theta\in\mathbb{R}$. 
Then, one has $\mathcal{H}(T)=\oplus_{n\in\mathbb{Z}}\mathcal{H}_n(T)$. Furthermore, by irreducibility of $T$, we have that $\textrm{dim}(\mathcal{H}_n(T)) = 0$ or $1$. In this way, one can construct an orthonormal basis $\{\phi_n\in\mathcal{H}_n(T):n\in\mathbb{Z}\}$ of $\mathcal{H}(T)$ such that $\phi_n\neq 0$ if and only if $\textrm{dim}(\mathcal{H}_n(T)) = 1$. 



Denote by $B_{n,m}(t) = \langle \phi_n, T(a(t))\phi_m\rangle$, where $a(t)=\left(\begin{array}{cc}e^t & 0 \\ 0 & e^{-t}\end{array}\right)$ is the (positive) diagonal 1-paramter subgroup of $SL(2,\mathbb{R})$. We will be interested in the decay properties of $B_{n,m}(t)$ as $t\to\infty$. To perform this study, we follow M. Ratner by making a series of preparations.

As it is shown in Lemma 2.1 of Ratner's paper \cite{Rt}, $y(t):=B_{n,m}(t)$ satisfies the following ODE 
$$y'' + 2y'-4\lambda y = f_1(t)+f_2(t)$$
where 
$$f_1(t)=(2e^{2t}\sinh(2t))^{-1}y'(t)$$
and 
$$f_2(t)=y(t)\left[\frac{2m(n-me^{-2t})}{\sinh(2t)}-\frac{(n-me^{-2t})^2}{\sinh^2(2t)}\right].$$ 

Furthermore, by the discussions after equation (2.12) and the equation (2.13) from Ratner's paper \cite{Rt}, one has 
$|y(t)|=|B_{n,m}(t)|\leq 1$ and $|y'(t)|=|B_{n,m}'(t)|\leq\sqrt{m^2-4\lambda}$. Hence,
$$|f_1(t)|\leq |(2e^{2t}\sinh(2t))^{-1}|\cdot\sqrt{m^2-4\lambda}$$
and 
$$|f_2(t)|\leq \left|\frac{2m(n-me^{-2t})}{\sinh(2t)}-\frac{(n-me^{-2t})^2}{\sinh^2(2t)}\right|.$$
Since 
$$|(2e^{2t}\sinh(2t))^{-1}|=e^{-4t}|(1-e^{-4t})^{-1}|\leq e^{-4t}\cdot(1-e^{-4})^{-1}$$ 
for every $t\geq 1$, we obtain that the constant $C_1$ appearing in equation (2.14) of Ratner's paper \cite{Rt} is 
\begin{equation}\label{e.C1-2-14}
C_1=(1-e^{-4})^{-1},
\end{equation}
i.e., 
\begin{equation}\label{e.2-14-1}
|f_1(t)|\leq C_1\sqrt{m^2-4\lambda}\cdot e^{-4t}
\end{equation}
with $C_1$ as above. 

Similarly, $\frac{2m(n-me^{-2t})}{\sinh(2t)}-\frac{(n-me^{-2t})^2}{\sinh^2(2t)}=\frac{1}{\sinh(2t)}[2m(n-me^{-2t})-\frac{(n-me^{-2t})^2}{\sinh(2t)}]$, so that 
$|f_2(t)|\leq\left|\frac{2m(n-me^{-2t})}{\sinh(2t)}-\frac{(n-me^{-2t})^2}{\sinh^2(2t)}\right|$ is bounded by the quantity $\frac{2}{(1-e^{-4})}e^{-2t}[2m(n-me^{-2t})-\frac{(n-me^{-2t})^2}{\sinh(2t)}]$ for every $t\geq1$. On the other hand, this last quantity is bounded by 
$$2C_1e^{-2t}\left[|2mn|+2e^{-2}m^2+2e^{-2}C_1n^2+2e^{-4}C_1|2mn|+2e^{-6}C_1m^2\right].$$
Because $|2mn|\leq m^2+n^2$ and $2e^{-2}+2C_1e^{-6}=2C_1e^{-2}$ (since $C_1=1/(1-e^{-4})$), we see that 
\begin{eqnarray*}
|f_2(t)|&\leq& \frac{2C_1}{e^{2t}}\left[\left(1+\frac{2}{e^2}+\frac{2C_1}{e^4}+\frac{2C_1}{e^6}\right)m^2+\left(1+\frac{2C_1}{e^2}+\frac{2C_1}{e^4}\right)n^2\right]\\
&=&2C_1e^{-2t}(1+2C_1e^{-2}+2C_1e^{-4})(m^2+n^2).
\end{eqnarray*}
In other words, the constant $C_2$ appearing in equation (2.14) of Ratner's paper \cite{Rt} is 
\begin{equation}\label{e.C2-2-14}
C_2=\frac{2}{1-e^{-4}}\left(1+\frac{2}{e^2(1-e^{-4})}+\frac{2}{e^4(1-e^{-4})}\right),
\end{equation}
i.e., 
\begin{equation}\label{e.2-14-2}
|f_2(t)|\leq C_2(m^2+n^2)e^{-2t}
\end{equation}
with $C_2$ as above. 

Next, we observe that the constant $C_1$ in equation (2.16) of Ratner's paper \cite{Rt} is slightly different that what she refers to as $C_1$ in equation (2.14). Indeed, by denoting the roots of the characteristic equation $x^2+2x-4\lambda=0$ of the ODE satisfied by $y(t):=B_{n,m}(t)$ by $r_1:=r_1(\lambda) := -1+\sqrt{1+4\lambda}$ and $r_2 := r_2(\lambda) := -1 - \sqrt{1+4\lambda}$, the fact that $|f_1(t)|\leq C_1e^{-4t}\sqrt{m^2-4\lambda}$ implies that 
\begin{eqnarray}\label{e.C1-2-16}
\left|\int_t^{\infty}e^{-r_1 s}f_1(s)ds\right|&\leq& C_1\sqrt{m^2-4\lambda}\int_t^{\infty}e^{(-\textrm{Re}(r_1)-4)s} ds \nonumber\\
&\leq& \frac{C_1}{3}\sqrt{m^2-4\lambda}\cdot e^{-3t}
\end{eqnarray}
because $\textrm{Re}(r_1)+2\geq 1$. However, the constant $C_2$ in Ratner's paper \cite{Rt} is the same for both (2.16) and (2.14):
\begin{eqnarray}\label{e.C2-2-16}
\left|\int_t^{\infty}e^{-r_1 s}f_2(s)ds\right|&\leq& C_2(m^2+n^2)\int_t^{\infty}e^{(-\textrm{Re}(r_1)-2)s} ds\nonumber\\
&\leq& C_2(m^2+n^2)e^{-t}
\end{eqnarray}
because $\textrm{Re}(r_1)+2\geq 1$.

Concluding our series of preparations, we recall the definitions of the following two functions
$$A_1(t):=\int_1^t e^{(r_1-r_2)s}\left(\int_s^{\infty} e^{-r_1 u}f_1(u)\,du\right)\,ds$$
and 
$$A_2(t):=\int_1^t e^{(r_1-r_2)s}\left(\int_s^{\infty} e^{-r_1 u}f_2(u)\,du\right)\,ds$$
introduced after equation (2.18) of Ratner's paper \cite{Rt}. These functions appear naturally in our context because the ODE verified by $y(t)=B_{n,m}(t)$ can be rewritten as $(D-r_1)(D-r_2)y = f_1(t)+f_2(t):=f(t)$ where $D$ is the differentiation operator (with respect to $t$). Thus, since $(D-r_1)y=e^{r_1 t}D(e^{-r_1t}y)$, we have $e^{r_1t}D(e^{-r_1t}(D-r_2)y) = f(t)$ and, hence, 
$$e^{(r_2-r_1)t}D(e^{-r_2t}y)=-\int_t^{\infty}e^{-r_1 s}f(s)\,ds+P_1$$
where $P_1$ is a constant. In particular, we can write 
\begin{equation}\label{e.2-17-eq}
y(t) = -e^{-t}\int_1^t\left(\int_s^{\infty}e^u\,f(u)\,du\right)\,ds + P_1te^{-t} + P_2 e^{-t}
\end{equation}
if $r_1=r_2$, and
\begin{eqnarray}\label{e.2-17-neq}
y(t) &=& e^{r_2 t}A(t) + e^{r_2 t}\left[P_1\int_1^t e^{(r_1-r_2)s}\,ds + P_2\right] \nonumber \\ 
&=& e^{r_2 t}A(t) + \frac{P_1}{2\sqrt{1+4\lambda}}e^{r_1 t} + \left(P_2 - \frac{e^{2\sqrt{1+4\lambda}}P_1}{2\sqrt{1+4\lambda}}\right)e^{r_2 t}
\end{eqnarray} 
if $r_1\neq r_2$, where $A(t):=A_1(t)+A_2(t)$. Moreover, by using these equations, and the fact that $y(t) = B_{n,m}(t)\to 0$ as $t\to\infty$ (a consequence of the non-triviality 
of $T$, that is, it has no invariant $T$-invariant vectors), we can deduce that 
\begin{equation}\label{e.2-21-P1}
P_1 = \int_1^{\infty}e^{-r_1s} f(s)\,ds - r_2 e^{-r_1}y(1)+e^{-r_1}y'(1)
\end{equation}
and 
\begin{equation}\label{e.2-21-P2}
P_2 = y(1) e^{-r_2}
\end{equation}

Finally, from the estimates \eqref{e.2-14-1}, \eqref{e.2-14-2} above, and the facts $\textrm{Re}(r_1)-\textrm{Re}(r_2)\geq 0$ and $\textrm{Re}(r_1)+2\geq 1$, we can estimate:
\begin{eqnarray}\label{e.C1bar-2-19}
&&|e^{r_2t}A_1(t)| \\ &&= |e^{r_2t}\int_1^t e^{(r_1-r_2)s}\int_s^{\infty}e^{-r_1 u}f_1(u)du\,ds| \nonumber\\ 
&&\leq C_1\sqrt{m^2-4\lambda} \,\, e^{t\textrm{Re}(r_2)}\int_1^t e^{(\textrm{Re}(r_1)-\textrm{Re}(r_2))s}\int_s^{\infty}e^{(-\textrm{Re}(r_1)-4)u}du \, ds\nonumber\\
&&\leq\frac{C_1}{3}\sqrt{m^2-4\lambda} \,\, e^{t\textrm{Re}(r_2)} e^{(\textrm{Re}(r_1)-\textrm{Re}(r_2))t}\int_1^t e^{(-\textrm{Re}(r_1)-4)s}ds\nonumber\\
&&\leq \frac{C_1}{9e^3}\sqrt{m^2-4\lambda} \,\, e^{t\textrm{Re}(r_1)}\nonumber
\end{eqnarray}
and 
\begin{eqnarray}\label{e.C2bar-2-20}
&&|e^{r_2t}A_2(t)|\\ &&= |e^{r_2t}\int_1^t e^{(r_1-r_2)s}\int_s^{\infty}e^{-r_1 u}f_2(u)du\,ds| \nonumber\\ 
&&\leq C_2(m^2+n^2)e^{t\textrm{Re}(r_2)}\int_1^t e^{(\textrm{Re}(r_1)-\textrm{Re}(r_2))s}\int_s^{\infty} e^{(-\textrm{Re}(r_1)-2)u}du\,\,ds\nonumber\\
&&\leq C_2(m^2+n^2)e^{t\textrm{Re}(r_1)}\int_1^t e^{(-\textrm{Re}(r_1)-2)s}ds\nonumber\\
&&\leq \frac{C_2}{e}(m^2+n^2)e^{t\textrm{Re}(r_1)}\nonumber
\end{eqnarray}
Thus, we can take 
\begin{equation}\label{e.Cbar}
\bar{C}_1=C_1/9e^3 \quad \quad \textrm{and} \quad\quad \bar{C}_2=C_2/e
\end{equation}
in equations (2.19) and (2.20) of Ratner's paper \cite{Rt}.

After these preparations, we are ready to pass to the next section, where we render more explicitly the constants appearing in Lemma 2.2 of Ratner's paper \cite{Rt} about the speed of decay of the matrix coefficients $B_{n,m}(t)$ as $t\to\infty$.  

\section{Decay of matrix coefficients of $SL(2,\mathbb{R})$-representations}

By following closely the proof of Lemma 2.2 of Ratner's paper \cite{Rt}, we show the following explicit variant of it:
\begin{lemma}\label{l.lemma2-2}For $t\geq1$, $n,m\in\mathbb{Z}$, 
$$|B_{n,m}(t)|\leq (\bar{K}_{\lambda}(m^2+n^2)+\tilde{K}_{\lambda})\cdot b_{\lambda}(t),$$
where 
\begin{itemize}
\item $b_{\lambda}(t)=te^{-t}$ if $\lambda\leq-1/4$;
\item $b_{\lambda}(t)=te^{r_1t}$ if $-1/4<\lambda<0$;
\item $b_{\lambda}(t)=te^{-2t}$ if $0\leq\lambda$;
\end{itemize}
and
$$\bar{K}_\lambda=\left\{\begin{array}{cl} 4C_1/9e^3+2C_2/e+e & \textrm{if }\lambda\leq-1/4\\ 
4C_1/9e^3+2C_2/e+e & \textrm{if }-1/4<\lambda<0\\ 
(C_1+C_2)/2 & \textrm{if } 0\leq\lambda\end{array}\right.,$$
$$\tilde{K}_{\lambda}=\left\{\begin{array}{cl} (1+2\sqrt{2})e + (32+\sqrt{2})C_1^2/3e^3  & \textrm{if }\lambda\leq-1/4\\ 
3e + e^2 + 4C_1/9e^3 & \textrm{if }-1/4<\lambda<0\\ 
e^2& \textrm{if } 0\leq\lambda\end{array}\right..$$
with the constants $C_1$ and $C_2$ given by \eqref{e.C1-2-14} and \eqref{e.C2-2-14} above.
\end{lemma} 

\begin{remark} In Ratner's article \cite{Rt}, the function $b_{\lambda}(t)$ is slightly different from the one above (when $\lambda<0$): indeed, in this paper, 
$$b_{\lambda}(t) = \left\{ \begin{array}{cl} \min\{te^{-t}, e^{-t}(1+\sqrt{1/|1+4\lambda|})\} & \textrm{if } \lambda\leq -1/4, \\ 
\min\{te^{r_1t},e^{r_1 t}\sqrt{1/|1+4\lambda|}\} & \textrm{if }-1/4<\lambda<0 \end{array} \right. .$$ In particular, this allows us to gain over the factor of $t$ (in front of the exponential functions $e^{-t}$, $e^{r_1t}$) when $\lambda$ is not close to $-1/4$ at the cost of permitting larger constants. However, since we had in mind the idea of getting uniform constants regardless of $\lambda$ and the factor of $t$ does not seem very substantial, we decided to neglect this issue by sticking to the function $b_{\lambda}(t)$ as defined in Lemma \ref{l.lemma2-2} above.  
\end{remark}

\begin{proof} We begin with the case $\lambda=-1/4$, i.e., $r_1=r_2=-1$. From \eqref{e.2-17-eq}, we know that 
$$y(t)=-e^{-t}\int_1^t\left(\int_s^{\infty}e^uf(u)du\right)ds+P_1te^{-t}+P_2e^{-t}.$$
Since, by definition, $f(t)=f_1(t)+f_2(t)$, we can apply \eqref{e.2-14-1}, \eqref{e.2-14-2} above to obtain
\begin{eqnarray}
|y(t)|&\leq& C_1\sqrt{m^2+1}\cdot e^{-t}\int_1^t\int_s^{\infty}e^{-3u}du\,\,ds \nonumber\\&+& C_2 (m^2+n^2)e^{-t}\int_1^t\int_s^{\infty}e^{-u}du\,\,ds \nonumber\\ &+&|P_1|te^{-t}+|P_2|e^{-t}.\nonumber
\end{eqnarray}
On the other hand, using that $|y(1)|\leq 1$, $|y'(1)|\leq\sqrt{m^2-4\lambda}$, the equations \eqref{e.C1-2-16}, \eqref{e.C2-2-16}, \eqref{e.2-21-P1}, \eqref{e.2-21-P2} above, 
and the fact that $r_1=r_2=-1$ in the present case, we get
\begin{eqnarray}
|y(t)| &\leq& \frac{C_1}{9e^3}\sqrt{m^2+1}\cdot e^{-t} + \frac{C_2}{e}(m^2+n^2)e^{-t} \nonumber \\
&+&te^{-t}\left[\frac{C_1}{3e^3}\sqrt{m^2+1}+\frac{C_2}{e}(m^2+n^2)+e+e\sqrt{m^2+1}\right] + e\cdot e^{-t}\nonumber
\end{eqnarray}
Since $\sqrt{m^2+1}\leq |m|+1\leq m^2+n^2+1$ and $e^{-t}\leq te^{-t}$ (because $t\geq 1$), we conclude that 
\begin{equation}\label{e.case1}
|y(t)|\leq te^{-t}\left[\bar{K}_{[\lambda=-1/4]}(m^2+n^2)+\tilde{K}_{[\lambda=-1/4]}\right]
\end{equation}
where 
$$\bar{K}_{[\lambda=-1/4]}=\frac{4C_1}{9e^3}+\frac{2C_2}{e}+e$$
and
$$\tilde{K}_{[\lambda=-1/4]}=\frac{4C_1}{9e^3}+3e$$

\bigskip

Next, we notice that, when $r_1\neq r_2$, by \eqref{e.2-17-neq}, and \eqref{e.C1bar-2-19}, \eqref{e.C2bar-2-20} above,  
\begin{eqnarray}\label{e.star}
|y(t)|&\leq& \frac{C_1}{9e^3}\sqrt{m^2-4\lambda}\cdot e^{t\textrm{Re}(r_1)}+ \frac{C_2}{e}(m^2+n^2)\cdot e^{t\textrm{Re}(r_1)} \\ 
&+& |P_1|\cdot te^{t\textrm{Re}(r_1)} + |P_2|\cdot e^{t\textrm{Re}(r_2)}\nonumber
\end{eqnarray}

If $-1/2\leq\lambda<-1/4$, we have that $\textrm{Re}(r_1)=\textrm{Re}(r_2)=-1$, $|r_2|\leq\sqrt{2}$ and $\sqrt{m^2-4\lambda}\leq\sqrt{m^2+2}\leq m^2+n^2+\sqrt{2}$, so that \eqref{e.2-21-P1}, \eqref{e.2-21-P2} and the equations \eqref{e.C1-2-16}, \eqref{e.C2-2-16} and \eqref{e.star} above imply
$$|P_1|\leq\frac{C_1}{3e^3}(m^2+n^2+\sqrt{2})+\frac{C_2}{e}(m^2+n^2)+\sqrt{2}\cdot e+e(m^2+n^2+\sqrt{2}),$$ 
$$|P_2|\leq e$$
and, \emph{a fortiori},
\begin{eqnarray}\label{e.case2}
|y(t)|&\leq&\left(\frac{C_1}{9e^3}+\frac{C_2}{e}\right)(m^2+n^2)te^{-t}+\frac{\sqrt{2}C_1}{9e^3}te^{-t}+(|P_1|+|P_2|)te^{-t}\nonumber\\
&\leq&te^{-t}\left(\bar{K}_{[-1/2\leq\lambda<-1/4]}(m^2+n^2)+\tilde{K}_{[-1/2\leq\lambda<-1/4]}\right)
\end{eqnarray}
where 
$$\bar{K}_{[-1/2\leq\lambda<-1/4]}=\frac{4C_1}{9e^3}+\frac{2C_2}{e}+e$$
and
$$\tilde{K}_{[-1/2\leq\lambda<-1/4]}=\frac{4\sqrt{2}C_1}{9e^3}+(1+2\sqrt{2})e.$$

\medskip

If $-1/4<\lambda<0$, we have that $0<\sqrt{1+4\lambda}<1$, so that $\textrm{Re}(r_1)=r_1=-1+\sqrt{1+4\lambda}\in(-1,0)$, $\textrm{Re}(r_2)=r_2=-1-\sqrt{1+4\lambda}\in (-2,-1)$, $|r_2|=1+\sqrt{1+4\lambda}\in(1,2)$ and $\sqrt{m^2-4\lambda}\leq\sqrt{m^2+1}\leq m^2+n^2+1$. Putting this into \eqref{e.2-21-P1}, \eqref{e.2-21-P2}, and the equations \eqref{e.C1-2-16}, \eqref{e.C2-2-16}, and \eqref{e.star} above, we get
$$|P_1|\leq\frac{C_1}{3e^3}(m^2+n^2+1)+\frac{C_2}{e}(m^2+n^2)+2e+e(m^2+n^2+1),$$
$$|P_2|\leq e^2,$$
and 
\begin{eqnarray}\label{e.case3}
|y(t)|&\leq& \left(\frac{C_1}{9e^3}+\frac{C_2}{e}\right)(m^2+n^2)te^{tr_1}+\frac{C_1}{9e^3}te^{tr_1}+(|P_1|+|P_2|)te^{tr_1}\nonumber\\
&\leq&te^{tr_1}(\bar{K}_{[-1/4<\lambda<0]}(m^2+n^2)+\tilde{K}_{[-1/4<\lambda<0]})
\end{eqnarray}
where 
$$\bar{K}_{[-1/4<\lambda<0]}=\frac{4C_1}{9e^3}+\frac{2C_2}{e}+e$$
and 
$$\tilde{K}_{[-1/4<\lambda<0]}=\frac{4C_1}{9e^3}+3e+e^2.$$

\bigskip

Now we pass to the case $\lambda<-1/2$. In this situation, $\sqrt{m^2-4\lambda}$ is not bounded, so we can't control $e^{r_2t}A_1(t)$ by using \eqref{e.C1bar-2-19}. So, we follow the arguments in page 281 of Ratner's paper \cite{Rt}. Recall that
$$A_1(t)=\int_1^t e^{(r_1-r_2)s}\left(\int_s^{\infty}e^{-r_1u} f_1(u)\,du\right)\,ds$$
and  
$$I(s):=\int_s^{\infty} e^{-r_1u}f_1(u) du=2\int_s^{\infty}e^{(-r_1-2)u}\frac{y'(u)}{\sinh(2u)}du.$$
Define $J(s):=\int_u^{\infty}y'(v)/\sinh(2v)\, dv$. By integration by parts, $J(u)=\frac{y(u)}{\sinh(2u)}+2\int_u^{\infty}y(v)\frac{\cosh(2v)}{\sinh^2(2v)}\,dv$,
$$I(s)=2\left[e^{(-r_1-2)s} J(s) + (r_1+2)\int_s^{\infty} e^{(-r_1-2)u} J(u)\, du\right]$$
and 
$A_1(t)=2(F_1(t)+F_2(t))$, where
$$F_1(t) = \int_1^t e^{(-r_2-2)s} J(s)\,ds$$
and
$$F_2(t) = (r_1+2)\int_1^t e^{(r_1-r_2)s}\left(\int_s^{\infty} e^{(-r_1-2)u}J(u)\,du\right)\,ds.$$

It follows that 
\begin{eqnarray}
|J(u)|&\leq&\frac{2}{(1-e^{-4})}e^{-2u} + 2\frac{(1+e^{-4})}{(1-e^{-4})}\int_u^{\infty}\frac{dv}{\sinh(2v)}\nonumber \\
&\leq& \frac{2}{(1-e^{-4})}e^{-2u} + 2\frac{(1+e^{-4})}{(1-e^{-4})^2}e^{-2u}\nonumber\\
&\leq& Q_1\cdot e^{-2u}\nonumber
\end{eqnarray}
where $Q_1=4/(1-e^{-4})^2=4C_1^2$. That is, we can take $Q_1=4C_1^2$ in the equation (2.22) of Ratner's paper \cite{Rt}. Also, since $\textrm{Re}(r_2)=-1$, we get
$$|e^{r_2t}F_1(t)|=\left|e^{r_2t}\int_1^t e^{-(r_2+2)s}J(s) ds\right|\leq Q_1e^{-t}\int_1^t e^{-3s}ds\leq Q_2e^{-t}$$
with $Q_2=4C_1^2/3e^3$, that is, this constant $Q_2$ works in equation (2.24) of Ratner's paper \cite{Rt}. Finally, by integrating by parts, 
$$e^{r_2t}F_2(t) = \frac{e^{r_2 t}(r_1+2)}{r_1-r_2}\left(\left[e^{(r_1-r_2)s}\int_s^{\infty} e^{(-r_1-2)u}J(u)\,du\right]_1^t + F_1(t)\right)$$
On the other hand, since $\lambda<-1/2$, one has $\left|\frac{r_1+2}{r_1-r_2}\right|=\left|\frac{1+\sqrt{1+4\lambda}}{2\sqrt{1+4\lambda}}\right|\leq 1$. By combining these facts, we see that 
$$|e^{r_2t}F_2(t)|\leq \frac{2Q_1}{3e^3}e^{-t}+Q_2e^{-t}=Q_3e^{-t}$$
where $Q_3=Q_1/e^3$. 

Thus, using these estimates to control $e^{r_2t}A_1(t)$ and the estimate~\eqref{e.C2bar-2-20} above to control $e^{r_2t}A_2(t)$, we obtain
\begin{eqnarray}
|e^{r_2t}A(t)|&\leq& |e^{r_2t}A_1(t)|+|e^{r_2t}A_2(t)|\nonumber\\ 
&\leq& 2|e^{r_2t}F_1(t)|+2|e^{r_2t}F_2(t)|+|e^{r_2t}A_2(t)|\nonumber\\
&\leq& 2(Q_2+Q_3)e^{-t}+\frac{C_2}{e}(m^2+n^2)e^{-t}\nonumber\\
&=& (\tilde{Q}+\bar{Q}(m^2+n^2))e^{-t}\nonumber
\end{eqnarray}
where $\tilde{Q}=2(Q_2+Q_3)=32C_1^2/3e^3$ and $\bar{Q}=\bar{C}_2=C_2/e$.

The second step in the analysis for the case $\lambda<-1/2$ is the control of the quantities $|P_1/2\sqrt{1+4\lambda}|$ and $|P_2-e^{2\sqrt{1+4\lambda}}P_1/2\sqrt{1+4\lambda}|$. Since $r_1=-1+i\sqrt{|1+4\lambda|}$, $r_2=-1-i\sqrt{|1+4\lambda|}$ and $|y(1)|\leq 1$, $|y'(1)|\leq\sqrt{m^2-4\lambda}$, we can estimate the first quantity as follows:
\begin{eqnarray}
\left|\frac{P_1}{2\sqrt{1+4\lambda}}\right|&\leq& \frac{1}{2\sqrt{|1+4\lambda}|}\left(\left|\int_1^{\infty}e^{-r_1s}f_1(s)ds\right| + \left|\int_1^{\infty}e^{-r_1s}f_2(s)ds\right|\right) \nonumber\\
&+&\frac{1}{2\sqrt{|1+4\lambda|}}\left(|r_2e^{-r_1}y(1)|+|e^{-r_1}y'(1)|\right)\nonumber\\
&\leq& \frac{1}{2\sqrt{|1+4\lambda}|}\left(\frac{C_1}{3e^3}\sqrt{m^2-4\lambda}+\frac{C_2}{e}(m^2+n^2)\right)\nonumber\\
&+& \frac{1}{2\sqrt{|1+4\lambda}|}\left(\sqrt{1+|1+4\lambda|}\cdot e+\sqrt{m^2-4\lambda}\cdot e\right)\nonumber\\
&=& \frac{1}{2\sqrt{|1+4\lambda}|}\left(\frac{C_1}{3e^3}\sqrt{m^2+|1+4\lambda|+1}+\frac{C_2}{e}(m^2+n^2)\right)\nonumber\\
&+& \frac{e}{2\sqrt{|1+4\lambda}|}\left(\sqrt{1+|1+4\lambda|}+\sqrt{m^2+|1+4\lambda|+1}\right)\nonumber\\
&\leq& \frac{1}{2\sqrt{|1+4\lambda}|}\left(\frac{C_1}{3e^3}+\frac{C_2}{e}+e\right)(m^2+n^2)\nonumber\\
&+& \frac{\sqrt{1+|1+4\lambda|}}{2\sqrt{|1+4\lambda}|}\left(\frac{C_1}{3e^3}+2e\right)\nonumber\\
&\leq& \frac{1}{2}\left(\frac{C_1}{3e^3}+\frac{C_2}{e}+e\right)(m^2+n^2)+\frac{1}{\sqrt{2}}\left(\frac{C_1}{3e^3}+2e\right).\nonumber
\end{eqnarray}
\smallskip
Here, we used that $\lambda<-1/2$ (so that $|1+4\lambda|>1$) and $\sqrt{(1+x)/x}<\sqrt{2}$ whenever $x>1$. Similarly, we can estimate the second quantity as follows:
$$
\left|P_2-\frac{e^{2\sqrt{1+4\lambda}}}{2\sqrt{1+4\lambda}}P_1\right| \leq |P_2|+\left|\frac{P_1}{2\sqrt{1+4\lambda}}\right|\leq \bar{Q}_1(m^2+n^2)+\tilde{Q}_1,
$$
where $\bar{Q}_1=\frac{1}{2}\left(\frac{C_1}{3e^3}+\frac{C_2}{e}+e\right)$ and $\tilde{Q}_1=\frac{1}{\sqrt{2}}\left(\frac{C_1}{3e^3}+2e\right)+e$. 
\medskip
Inserting these estimates above into \eqref{e.2-17-neq}, we deduce that 
\begin{eqnarray}\label{e.case4}
|y(t)|&\leq& |e^{r_2t}A(t)|+\left|\frac{P_1}{2\sqrt{1+4\lambda}}\right|\cdot |e^{r_1t}|+\left|P_2-\frac{e^{2\sqrt{1+4\lambda}}}{2\sqrt{1+4\lambda}}P_1\right|\cdot |e^{r_2t}|\nonumber\\
&\leq& \left((\bar{Q}+2\bar{Q}_1)(m^2+n^2)+(\tilde{Q}+2\tilde{Q}_1-e)\right)e^{-t}\nonumber\\
&=& \left(\bar{K}_{[\lambda<-1/2]}(m^2+n^2)+\tilde{K}_{[\lambda<-1/2]})\right)e^{-t}
\end{eqnarray}
where 
$$\bar{K}_{[\lambda<-1/2]}=\frac{C_1}{3e^3}+\frac{2C_2}{e}+e$$
and 
$$\tilde{K}_{[\lambda<-1/2]}=\frac{(32+\sqrt{2})C_1^2}{3e^3}+(1+2\sqrt{2})e.$$

\bigskip

Finally, we consider the case $\lambda\geq0$. We begin by estimating $|e^{r_2t}A_2(t)|$ and $e^{r_2t}A_1(t)$: using \eqref{e.2-14-1}, \eqref{e.2-14-2} above and $r_1=-1+\sqrt{1+4\lambda}\geq0$, $r_2=-1-\sqrt{1+4\lambda}\leq-2$, we obtain
\begin{eqnarray}
|e^{r_2t}A_2(t)|&\leq&C_2(m^2+n^2)e^{r_2t}\int_1^t e^{(r_1-r_2)s}\int_s^{\infty}e^{-r_1u}e^{-2u}du\,\,ds\nonumber\\
&\leq&\frac{C_2}{2}(m^2+n^2)e^{r_2t}\int_1^t e^{(-r_2-2)s}ds\nonumber\\
&\leq& \frac{C_2}{2}(m^2+n^2)te^{-2t}\nonumber
\end{eqnarray}
and
\begin{eqnarray*}
|e^{r_2t}A_1(t)|&\leq& \frac{C_1}{2}\sqrt{m^2-4\lambda}e^{r_2t}\int_1^t e^{-(r_2+2)s}ds\leq\frac{C_1}{2}|m| t e^{-2t}\\
&\leq&\frac{C_1}{2}(m^2+n^2)te^{-2t}.
\end{eqnarray*}
Thus, 
$$|e^{r_2t}A(t)|\leq |e^{r_2t}A_1(t)|+|e^{r_2t}A_2(t)|\leq\frac{C_1+C_2}{2}(m^2+n^2)te^{-2t},$$
so that we can take $\bar{C}=(C_1+C_2)/2$ and $\tilde{C}=0$ in the equation (2.28) of Ratner's paper \cite{Rt}. 

Next, we observe that $y(t)\to0$ when $t\to\infty$ and $r_1\geq 0$ imply $P_1=0$ and 
$$y(t)=e^{r_2t}A(t)+y(1)e^{-r_2}e^{r_2t}.$$

Therefore, from the previous discussion and $r_2+2\leq0$, it follows that 
\begin{eqnarray}\label{e.case5}
|y(t)|&\leq&\frac{C_1+C_2}{2}(m^2+n^2)te^{-2t}+e^{-r_2}e^{(r_2+2)t}e^{-2t}\nonumber\\
&\leq& (\bar{K}_{[\lambda\geq0]}(m^2+n^2)+\tilde{K}_{[\lambda\geq0]})te^{-2t}
\end{eqnarray}
where 
$$\bar{K}_{[\lambda\geq0]}=\frac{C_1+C_2}{2}$$
and
$$\tilde{K}_{[\lambda\geq0]}=e^2$$

\bigskip

At this stage, from \eqref{e.case1}, \eqref{e.case2}, \eqref{e.case3}, \eqref{e.case4}, \eqref{e.case5} above, we see that the proof of the desired lemma is complete.
\end{proof}

In next (and final) section, we apply Lemma \ref{l.lemma2-2} to derive explicit variants of Theorems 1 and 3 of Ratner's paper \cite{Rt}. To do so, we recall some notation already introduced in Section \ref{s.preliminaries}. We denote by 
$r(\theta)=\left(\begin{array}{cc}\cos\theta & \sin\theta \\ -\sin\theta & \cos\theta\end{array}\right)\in SO(2,\mathbb{R})$, $\theta\in\mathbb{R}$. Given an unitary $SL(2,\mathbb{R})$-representation $T$, we denote by $K(T,3)$ the set of vectors $v\in\mathcal{H}(T)$ such that $\theta\mapsto T(r(\theta))v$ is $C^3$. Finally, if the map $\theta\mapsto T(r(\theta))v$ is $C^1$, we denote by 
$$L_Wv := \lim\limits_{\theta\to 0}\frac{T(r(\theta))v-v}{\theta}$$
the Lie derivative of $v$ along the direction of $W=\left(\begin{array}{cc}0 & 1 \\ -1 & 0\end{array}\right)$ of the infinitesimal generator of the rotation group $SO(2,\mathbb{R})=\{r(\theta):\theta\in\mathbb{R}\}$.

In particular, in the case of an irreducible unitary $SL(2,\mathbb{R})$-representation $T$, since $T(r(\theta))\phi_n=e^{in\theta}\phi_n$ when $\phi_n\in\mathcal{H}_n(T)$, we have that 
$$L_W\phi_n = i n\phi_n$$ 
for every $n\in\mathbb{Z}$. 

\section{Proof of Theorems \ref{t.theorem1} and \ref{t.theorem3}}

In this short section, we indicate how Lemma \ref{l.lemma2-2} can be used to prove Theorems \ref{t.theorem1} and \ref{t.theorem3} (whose respective statements are recalled below).

\begin{theorem}\label{t.theorem1'}Let $T$ be a non-trivial irreducible unitary representation of $SL(2,\mathbb{R})$ in $\mathcal{H}(T)$ and let $\lambda=\lambda(T)$. Let $v,w\in K(T,3)$ and $B(t)=\langle v,w\circ a(t)\rangle$. Then, for all $t\geq 1$, 
\begin{eqnarray*}
|B(t)|&\leq&\sqrt{2\zeta(2)}\cdot\bar{K}_{\lambda}\cdot\|L_W^3v\|\cdot(\|w\|+\sqrt{2\zeta(6)}\|L_W^3w\|)\cdot b_{\lambda}(t) \\
&+&\sqrt{2\zeta(2)}\cdot\bar{K}_{\lambda}\cdot(\|v\|+\sqrt{2\zeta(6)}\|L_W^3v\|)\cdot\|L_W^3w\|\cdot b_{\lambda}(t)\\
&+&\tilde{K}_{\lambda}\cdot(\|v\|+\sqrt{2\zeta(6)}\|L_W^3v\|)\cdot(\|w\|+\sqrt{2\zeta(6)}\|L_W^3w\|)\cdot b_{\lambda}(t)
\end{eqnarray*}
where $\bar{K}_{\lambda}$, $\tilde{K}_{\lambda}$ and $b_{\lambda}(t)$ are as in Lemma~\ref{l.lemma2-2}. 
\end{theorem}

\begin{proof} Following the proof of Theorem 1 of Ratner's paper \cite{Rt} (at page 283), we write 
$$v=\sum\limits_{n\in\mathbb{Z}}c_n\phi_n, \quad w=\sum\limits_{n\in\mathbb{Z}}d_n\phi_n$$
with $c_n=\langle v,\phi_n\rangle$, $d_n=\langle w,\phi_n\rangle$ (and $\phi_n\in\mathcal{H}_n(T)$,  $n\in\mathbb{Z}$) as in page 276 of this paper. We have 
$$B(t)=\sum\limits_{n,m\in\mathbb{Z}} c_n d_m B_{n,m}(t)$$
so that 
$$|B(t)|\leq b_{\lambda}(t)\sum\limits_{n,m\in\mathbb{Z}}|c_n|\cdot|d_m|\cdot(\bar{K}_{\lambda}(m^2+n^2)+\tilde{K}_{\lambda})$$
by Lemma~\ref{l.lemma2-2}.

On the other hand, since $L_W^3\phi_n = -i n^3\phi_n$ for all $n\in\mathbb{Z}$, we know that 
\begin{eqnarray*}
\sum\limits_{n\in\mathbb{Z}}|c_n|&\leq& |c_0|+\left(\sum\limits_{n\in\mathbb{Z}-\{0\}}n^6|c_n|^2\right)^{\frac{1}{2}}\left(\sum\limits_{n\in\mathbb{Z}-\{0\}}\frac{1}{n^6}\right)^{\frac{1}{2}}\\ &\leq&\|v\| + \sqrt{2\zeta(6)}\cdot\|L_W^3v\|,
\end{eqnarray*}
\begin{eqnarray*}
\sum\limits_{m\in\mathbb{Z}}|d_m|&\leq& |d_0|+\left(\sum\limits_{m\in\mathbb{Z}-\{0\}}m^6|d_m|^2\right)^{\frac{1}{2}}\left(\sum\limits_{m\in\mathbb{Z}-\{0\}}\frac{1}{m^6}\right)^{\frac{1}{2}}\\ &\leq&\|w\| + \sqrt{2\zeta(6)}\cdot\|L_W^3w\|,
\end{eqnarray*}
\begin{eqnarray*}
\sum\limits_{n\in\mathbb{Z}}|c_n|\cdot n^2&\leq& \left(\sum\limits_{n\in\mathbb{Z}-\{0\}}n^6|c_n|^2\right)^{\frac{1}{2}}\left(\sum\limits_{n\in\mathbb{Z}-\{0\}}\frac{1}{n^2}\right)^{\frac{1}{2}}\\ &\leq&\sqrt{2\zeta(2)}\cdot\|L_W^3v\|,
\end{eqnarray*}
and
\begin{eqnarray*}
\sum\limits_{m\in\mathbb{Z}}|d_m|\cdot m^2&\leq& \left(\sum\limits_{m\in\mathbb{Z}-\{0\}}m^6|d_m|^2\right)^{\frac{1}{2}}\left(\sum\limits_{m\in\mathbb{Z}-\{0\}}\frac{1}{m^2}\right)^{\frac{1}{2}}\\ &\leq&\sqrt{2\zeta(2)}\cdot\|L_W^3w\|,
\end{eqnarray*}
The desired result follows.
\end{proof}

\begin{theorem}\label{t.theorem3'}Let $T$ be an unitary representation of $SL(2,\mathbb{R})$ having no non-zero invariant vectors in $\mathcal{H}(T)$. Write $\Lambda=\Lambda(\Omega_T)$ the spectrum of the Casimir operator and
$$A(T)=\Lambda\cap(-1/4,0).$$
If $A(T)\neq\emptyset$, let $\beta(T)=\sup A(T)$ and $\sigma(T)=-1+\sqrt{1+4\beta(T)}$. Assume that $\beta(T)<0$ when $A(T)\neq\emptyset$. Let $B(t)=\langle v,w\circ a(t)\rangle$ with $v,w\in K(T,3)$. Then, for all $t\geq 1$,
\begin{eqnarray*}
|B(t)|&\leq&\sqrt{2\zeta(2)}\cdot\bar{K}\cdot\|L_W^3v\|\cdot(\|w\|+\sqrt{2\zeta(6)}\|L_W^3w\|)\cdot b_T(t) \\
&+&\sqrt{2\zeta(2)}\cdot\bar{K}\cdot(\|v\|+\sqrt{2\zeta(6)}\|L_W^3v\|)\cdot\|L_W^3w\|\cdot b_T(t)\\
&+&\tilde{K}_T\cdot(\|v\|+\sqrt{2\zeta(6)}\|L_W^3v\|)\cdot(\|w\|+\sqrt{2\zeta(6)}\|L_W^3w\|)\cdot b_T(t)
\end{eqnarray*}
where $\bar{K}=\frac{4C_1}{9e^3}+\frac{2C_2}{e}+e$, 
$$\tilde{K}_T=\left\{\begin{array}{cl}\frac{(32+\sqrt{2})C_1^2}{3e^3}+(1+2\sqrt{2})e & \textrm{ if } A(T)=\emptyset, \\ 
\frac{4C_1}{9e^3}+3e+e^2 & \textrm{ if } A(T)\neq\emptyset\end{array}\right.,$$
$$b_T(t)=\left\{\begin{array}{cl}t\cdot e^{-t} & \textrm{ if } A(T)=\emptyset, \\ 
t\cdot e^{\sigma(T)t} & \textrm{ if } A(T)\neq\emptyset\end{array}\right.$$
\end{theorem}

\begin{proof} This is an immediate consequence of Theorem \ref{t.theorem1'} and the arguments from pages 285--286 of Ratner's paper \cite{Rt}.
\end{proof}




\begin{thebibliography}{99}
































\bibitem[EMV]{EMV} M.~Einsiedler, G.~Margulis and A.~Venkatesh, \emph{Effective equidistribution for closed orbits of semisimple 
groups on homogeneous spaces}, Invent. Math., \textbf{177} (2009), 137--212.

\bibitem[KM]{Kahn:Markovic} J. Kahn and V. Markovic, \emph{Immersing almost geodesic surfaces in a closed hyperbolic three manifold}, 
Annals of Math., \textbf{175} (2012), no. 3, 1127--1190.

\bibitem[Kn]{Knapp} A. W. Knapp, \emph{Representation theory of semisimple groups. An overview based on examples}, Princeton Mathematical Series, 36. Princeton University Press, Princeton, NJ, 1986.  

\bibitem[LM]{Li:Margulis} H.~Li and G.~Margulis, \emph{Effective discreteness of the $3$-dimensional Markov spectrum}, in preparation.

\bibitem[M]{Margulis} G.~Margulis, \emph{Discrete subgroups and ergodic theory}. Number theory, trace formulas and discrete groups (Oslo, 1987), 377--398, Academic Press, Boston, MA, 1989.

\bibitem[Ma]{DM} C.~Matheus, \emph{Explicit constants in Ratner's estimates on rates of mixing of geodesic flows of hyperbolic surfaces}, post at the mathematical blog ``Disquisitiones Mathematicae'' (http://matheuscmss.wordpress.com/).

\bibitem[MS]{Matheus:Schmithuesen} C.~Matheus and G.~Schmith\"usen, \emph{Explicit Teichm\"uller curves with complementary series}, to appear in Bulletin de la SMF.

\bibitem[Mo]{Mohammadi} A.~Mohammadi, \emph{A special case of effective equidistribution with explicit constants}, 
Ergodic Theory Dynam. Systems, \textbf{32} (2012), no. 1, 237--247.

\bibitem[R]{Rt} M.~Ratner, \emph{The rate of mixing for geodesic and horocycle flows}, Ergodic Theory Dynam. Systems, \textbf{7} (1987), 267--288.
















\end{thebibliography}
\end{document}